\mag=\magstep1

\documentclass[10pt,a4paper,dviout]{amsart} 
\usepackage{amssymb,latexsym} 
\usepackage{color} 
\usepackage{graphicx}
\usepackage{comment}
\usepackage{datetime}
\usepackage{comment}

\input xy
\xyoption{all}

\usepackage{amsmath, amssymb, eucal, amscd, color, amsthm}

\def\R{\Bbb R}
\def\si{\sigma}

\def\cal{\mathcal}
\def\co{{\cal O}}
\def\Q{\Bbb Q}
\def\part{\partial}

\def\we{\wedge}
\def\e{\epsilon}
\def\dis{\displaystyle}

\def\P{\mathbb P}
\def\C{\mathbb C}
\def\s{{\square^n}}

\def\id{\operatorname{id}}
\def\codim{{{\rm codim}\,}}

\def\p1{\prec}

\def\<{\langle}
\def\>{\rangle}

\def\id{\operatorname{id}}
\newcommand{\sign}{\operatorname{sign}}
\newcommand{\Alt}{\operatorname{Alt}}

\newtheorem{theorem}{Theorem}[section] 
\newtheorem{proposition}[theorem]{Proposition} 
 
\newtheorem{definition}[theorem]{Definition} 
 
\newtheorem{remark}[theorem]{Remark} 
\newtheorem{lemma}[theorem]{Lemma} 
\newtheorem{lemma-definition}[theorem]{Lemma-Definition} 
\newtheorem{proposition-definition}[theorem]{Proposition-Definition}

\newcommand{\CC}{{\mathbb{C}}}
\newcommand{\PP}{{\mathbb{P}}}
\newcommand{\QQ}{{\mathbb{Q}}}

\newcommand{\ZZ}{{\mathbb{Z}}}

\newcommand\Spec{\operatorname{Spec}}

\newcommand{\al}{\alpha}  \newcommand{\ga}{\gamma} 
 \newcommand{\eps}{\epsilon} 
\newcommand{\Ga}{\Gamma}

\newcommand{\sq}{\square} 

\newcommand{\ot}{\otimes}

\newcommand{\mapright}[1]{%
  \smash{\mathop{%
    \hbox to 1cm{\rightarrowfill}}\limits^{#1} } } 

\newcommand{\smapr}[1]{%
  \smash{\mathop{%
    \hbox to 0.5cm{\rightarrowfill}}\limits^{#1} } } 
\newcommand{\maprb}[1]{%
  \smash{\mathop{%
    \hbox to 1cm{\rightarrowfill}}\limits_{#1} } } 
\newcommand{\mapleft}[1]{%
  \smash{\mathop{%
    \hbox to 1cm{\leftarrowfill}}\limits^{#1} } }

\newcommand{\maplb}[1]{%
  \smash{\mathop{%
    \hbox to 1cm{\leftarrowfill}}\limits_{#1} } }

\def\alt{\operatorname{alt}}

\makeatletter
  
  \@addtoreset{equation}{section}
 \makeatother

\makeindex
\usepackage{version}	
\begin{document}

\title{An application of a Hodge realization of Bloch-Kriz mixed Tate motives} 
\author{Kenichiro Kimura }

\maketitle
\begin{abstract} Beilinson and Deligne proved a weak version of Zagier's conjucture on special values of Dedekind
zeta functions assuming the existence of a category of mixed Tate motives which has certain properties.
We show that Bloch-Kriz category of mixed Tate motives together with a Hodge realization constructed in
\cite{part II} has the required properties.

\end{abstract}

\vskip 0.3cm


\setcounter{tocdepth}{3}


\thispagestyle{empty}
\setcounter{tocdepth}{1}

\section{Introduction}
In \cite{BD} a weak version of Zagier's conjecture on  special values of Dedekind
zeta functions is proved under certain assumptions which are stated bellow.
Let $F$ be a number field. 
\begin{itemize}
\item[(A)]  There exists  a Tannakian category $\cal T(F)$ over $\Q$.
\item[(B)]  There exists an object $\Q(1)$ of rank one on $\cal T(F)$.  Let $\Q(k)=\Q(1)^{\otimes k}$ for $k\in \ZZ.$
The simple objects of $\cal T(F)$ are $\Q(k)$ for $k\in \ZZ.$  They are pairwise non-isomorphic, and 
satisfy the identities
\[\text{Ext}^1(\Q(0),\Q(n))=0\,\,\text{ for }\,\,n\leq 0.\]
  There are isomorphisms
\[\text{Ext}^1(\Q(0),\Q(k))=K_{2k-1}(F)\otimes \Q\]
for $k\geq 1.$

\item[(C)]  For each embedding $\si$ of $F$ into $\CC$,  there exists an
exact tensor functor ``Hodge realization''  from $\cal T(F)$ to the category
of $\Q$-mixed Hodge structures.  One has real$(\Q(1))=\Q(1)$.

\item[(D)] The Hodge realization is compatible with
the regulator maps on  $K_{2k-1}(F)$ for $k\geq 1$.

\item[(E)]  For $a\in F-\{0,1\}$, there exists a projective system $M_k(a)$
of extensions of $\Q(0)$ by $\text{Sym}^{k-1}([a])(1)$ with
$M_1(a)=[1-a]$,  the Hodge realizations of which are the ones described in 
\cite{BD}.
\end{itemize}
In \cite{BK} a certain Tannakian category of mixed Tate motives over $F$ 
is constructed for any field $F$. This category is denoted by ${\rm MT_{BK}}$ in this paper.
They also construct Polylog objects of ${\rm MT_{BK}}$ which should 
satisfy the condition (E) if $F$ is a subfield of $\CC$.  In \cite{part II} we construct 
a Hodge realization functor of ${\rm MT_{BK}}$ when $F$ is a subfield of $\CC$.  The purpose of
this paper is to show that the category ${\rm MT_{BK}}$,  our Hodge realization and the Polylog 
objects of ${\rm MT_{BK}}$   satisfy the the conditions (A) to (E).

We should remark on a Hodge realization of ${\rm MT_{BK}}$ described in \cite{BK}.  There a Hodge
realization functor of ${\rm MT_{BK}}$  in terms  of a certain complex of topological chains
$\cal D \cal P$ is described. They do not define $\cal D \cal P$ explicitly,
but the properties which $\cal D \cal P$ should have are stated, and some explicit
chains which should be in $\cal D \cal P$ are described.  The Hodge realization of
Polylog motives is computed in terms of those ``should be'' elements of $\cal D \cal P$.
In this sense the Hodge realization of ${\rm MT_{BK}}$ in terms of $\cal D \cal P$
is not complete.  

\noindent The construction of  our Hodge realization is inspired by the one 
in terms of $\cal D \cal P$. In \cite{part II} we explicitly define a certain complex
$AC^\bullet$ of topological chains and prove that it has the necessary  properties.
In \cite{Ki} we give a description of the Abel-Jacobi maps on higher Chow cycles
in terms of $AC^\bullet$.  The compatibility of the Hodge realization  with
the regulator map follows from this. 

For a field $F$  there exists a triangulated category ${\rm DTM}(F)$ of mixed Tate motives in ${\rm DM_{gm}}(F)$
of Voevodsky's category.  When $F$ satisfies  the Beilinson-Soul\'{e} conjectures e.g. F is a number field,
Levine \cite{Lev2} defines a $t$-structure on ${\rm DTM} (F)$, and the heart ${\rm MT}(F)$ is an abelian category.
By a result of Spitzweck \cite{Sp} (see \cite{Lev3} Section 5 for an account)  it can be shown
that the category ${\rm MT_{BK}}$ is equivalent to ${\rm MT}(F)$.  The compatibility of our Hodge
realization of $\rm MT_{BK}$ and that of ${\rm MT}(F)$ is still to be proved.

For recent progress on Zagier's conjecture see for example \cite{Go}.

The organization of this paper is as follows.  In Section \ref{MT}  we give an overview
of the category ${\rm MT_{BK}}$.  We see that ${\rm MT_{BK}}$ satisfies the conditions (A) and (B).
In Section \ref{Hodger} we will give an overview of our Hodge realization of ${\rm MT_{BK}}$.
The conditions (C) and (D) are shown to hold.  Finally in Section \ref{polylog}
we recall the definition of Polylog motive from \cite{BK} and compute its
Hodge realization.  The condition (E) will be shown to hold.

\section{The category ${\rm MT_{BK}}$}
\label{MT}
In this section we give an overview of the category of mixed Tate motives
${\rm MT_{BK}}$ which is constructed in
\cite{BK}.   We will see that ${\rm MT_{BK}}$ satisfies conditions (A) and (B).
\subsection{Cycle complexes and a cdga $N$}

Let $F$ be a field.    
Following \cite{BK}, we recall that the cycle complex of $\Spec {F} $ may be viewed as a 
commutative differential graded algebra (cdga) over $\QQ$. 

Bloch defined the cycle complex for any quasi-projective variety, but we will restrict to the
 case of $\Spec {F}$. 
Let $\sq^n=\sq^n_{F}=(\PP^1_{F} -\{1\})^n$, which is isomorphic to the affine $n$-space 
as a variety.
Let  $(z_1, \cdots, z_n)$ be the inhomogeneous coordinates of $\sq^n$.
For $i=1,\cdots,n$ and $\al=0,\infty$,
we denote by $H_{i,\al}$ the subvariety of $\square^n$ defined by $z_i=\al$.  The intersection of several $H_{i,\al}$'s is called a {\it cubical face}.

For $n\ge 0$ and $r\ge 0$, let $Z^r(\Spec F, n)$\index{$Z(r, n)$} be the $\QQ$-vector space with basis irreducible 
closed subvarieties of $\sq^n$ of codimension 
$r$ which meet all the cubical faces properly.   
The cubical differential $\partial: Z^r(\Spec F, n)\to Z^r(\Spec F, n-1)$ is defined by the equality
\begin{equation}
\label{total face map definition} 
\part =\sum_{i=1}^n (-1)^{i-1}
(\part_{i,0}-\part_{i,\infty})
\end{equation}
where $\part_{i,\al}$ is the intersection with the cubical face $H_{i,\al}$. The group $G_n= \{\pm 1\}^n\rtimes S_n$\index{$G_n$} acts naturally on  $\s$ as follows. The subgroup
$\{\pm 1\}^n$ acts by the inversion of the coordinates $z_i$, and  the symmetric group $S_n$ acts by
permutation of $z_i$'s. This action induces an action of  $G_n$ on $Z^r(\Spec F, n)$.
Let $\sign: G_n \to \{\pm 1\}$ be the character which sends 
$(\eps_1, \cdots, \eps_n; \sigma)$ to $\eps_1\cdot \cdots\cdot\eps_n\cdot\sign(\sigma)$. 
The idempotent $\Alt=\Alt_n:=({1}/{|G_n|})\sum_{g\in G_n} \sign(g) g $ \index{$\Alt$}
in the group ring $\QQ[G_n]$ is called the alternating 
projector. For a $\Q[G_n]$-module M, 
the submodule
$$
M^{\alt}=\{\al \in M\mid \Alt \al=\al\}=\Alt(M)
$$
\index{$M^{\alt}$}
is called the alternating part of $M$. By Lemma 4.3 \cite{BK}, one knows that the cubical differential
$\part $ maps $Z^r(\Spec F,n)^{\rm alt}$ to $Z^r(\Spec F,n-1)^{\rm alt}$. 
For convenience let $ Z^r(\Spec {F}, n)=0$ if $n<0$. 
Product of cycles induces a map of complexes 
$\times : Z^r(\Spec F, n)\otimes Z^s(\Spec F, m)\to Z^{r+s}(\Spec F, n+m)$, $z\otimes w\mapsto z\times w$. 
This induces a map of complexes on alternating parts 
$$
Z^r(\Spec {F}, n)^{\rm alt}\otimes Z^s(\Spec {F}, m)^{\alt}\to Z^{r+s}(\Spec {F}, n+m)^{\rm alt}
$$
given by $z\otimes w\mapsto z\cdot w=\Alt (z\times w)$.
We set $N_r^i= Z^r(\Spec {F}, 2r-i)^{\rm alt}$ for $r\ge 0$ and $i\in \ZZ$
\index{$N_r^i$, $N^\bullet_r$}.


One thus has an associative product map
$$N_r^i\otimes N_s^j\to N_{r+s}^{i+j}\,,\qquad z\otimes w\mapsto z\cdot w\,,$$
One verifies that the product is 
graded-commutative:
$w\cdot z=(-1)^{ij}z\cdot w$ for $z\in N_r^i$ and $w\in N_s^j$. 

Let $\dis N=\underset{r\ge 0,\, i\in \ZZ}\oplus  N_r^i$, and $N_r=\underset{i\in \ZZ}\oplus N_r^i$. 
 The algebra $N$ with the above product and the differential $\part$
makes it a commutative differential graded algebra (cdga) over $\QQ$.

\noindent By definition, one has $N_0=\QQ$ and $1\in N_0$ is the unit for the product. 
Thus the projection $\epsilon : N\to N_0=\QQ$\index{augmentation $\epsilon$} is an augmentation, namely it is a map of cdga's and the composition 
with the unit map $\QQ\to N$ is the identity. Let  $r$ be a non-negative integer  and  $p$ be an integer. Then
by \cite{BK} Proposition 5.1 we have the folloing isomorphism.
\begin{equation}
\label{cohomN}
H^p({N_r}^\bullet)\simeq CH^r(\Spec F, 2r-p)=H^p_M (\Spec F, \Q(r)).
\end{equation}

\subsection{The bar complex}
\label{bar complex}
Let $M$ (resp. $L$) be a complex which is a differential left $N$-module  (resp. right  $N$-module).  We recall the definition
of the bar complex $B(L, N, M)$ (see \cite{EM}, \cite{Ha} for details.) 

Let $N_+= \oplus_{r>0} N_r$.  As a module, $B(L, N,M)$\index{$B(L, N,M)$} 
is equal to 
$L\otimes \bigl(\bigoplus_{s\ge 0} (\otimes ^s N_+)\bigr)\otimes M$, 
with the convention $(\otimes ^s N_+)=\QQ$ for $s=0$. 
An element $l\otimes (a_1\otimes \cdots \otimes a_s) \otimes m$  of 
 $L\otimes  (\otimes ^s N_+)\otimes M$ 
  is written as
$l[a_1|\cdots |a_s] m$ (for $s=0$, we write $l[\,\,]m$ for $l\otimes 1\otimes m $ in $L\otimes \QQ\otimes M$ ).

 The internal differential $d_I$ is defined by
\begin{align*}
 d_I&(l[a_1|\cdots |a_s] m)\\
=&dl  [a_1|\cdots |a_s]m\\
+&\sum_{i=1}^{s}
(-1)^i Jl[Ja_1|\cdots |Ja_{i-1}|da_{i}|\cdots |a_s] m
+(-1)^sJl[Ja_1|\cdots |Ja_s] dm
\end{align*}
where
$Ja=(-1)^{\deg a}a$.
The external differential $d_E$ is defined by
\begin{align*}
 d_E&( l [a_1|\cdots |a_s] m)\\
=&-(Jl)a_1[a_2|\cdots |a_s] m\\
+&\sum_{i=1}^{s-1}
(-1)^{i+1}Jl[Ja_1|\cdots |(Ja_{i})a_{i+1}|\cdots |a_s] m
\\
+&(-1)^{s-1}Jl[Ja_1|\cdots |Ja_{s-1}] a_sm.
\end{align*}
Then we have $d_Id_E+d_Ed_I=0$ and the map $d_E+d_I$ defines
a differential on $B(L, N,M)$. The degree of an element
$l[a_1|\cdots |a_s] m$
is defined to be $\sum_{i=1}^s \deg a_i +\deg l+\deg m-s$. 

If $L=\QQ$ and the right $N$-module structure is given by the augmentation
$\epsilon$,
the complex $B(L,N,M)$ is denoted by $B(N,M)$
and omit the first factor ``$1\otimes$''.
If $L=M=\Q$ with the $N$-module structure given by the augmentation $\e$,
we set 
$$
B(N):=B(\Q,N, \Q).
$$
We omit the first and the last tensor factor ``$1\otimes$'' and
``$\otimes 1$'' for
an element in $B(N)$. \index{$B(N), B(N)_r$}

The complex $B(N)$ is graded by non-negative integers as a complex: $B(N)=\oplus_{r\ge 0}B(N)_r$  where 
$B(N)_0 = \QQ$
and  for $r > 0$
$$
B(N)_r=\bigoplus_{r_1+\cdots +r_s=r,\,r_i> 0}
N_{r_1}\otimes \cdots \otimes N_{r_s}.
$$
Let $\Delta:B(N)\to B(N)\otimes B(N)$
\index{$\Delta:B(N)\to B(N)\otimes B(N)$}
be the map given by
$$
\Delta([a_1|\cdots |a_s])
=
\sum_{i=0}^s\big([a_1|\cdots |a_i]\big)\otimes
\big([a_{i+1}|\cdots |a_s]\big).
$$
and $e : B(N) \to \QQ$ be the projection to $B(N)_0$ . These are maps of complexes, and they satisfy
coassociativity $(\Delta \otimes 1)\Delta = (1\otimes \Delta)\Delta$ and counitarity 
$(1\otimes e)\Delta = (e\otimes 1)\Delta = \text{id}$. In other words
$\Delta$ is a coproduct on $B(N)$ with counit $e$.
In addition, the shuffle product (see e.g. \cite{EM}) makes $B(N)$ a
differential graded algebra with unit $\QQ = B(N)_0\subset B(N)$. The shuffle product is graded-commutative. Further, the
maps $\Delta$ and $e$ are compatible with product and unit. We summarize:
\begin{enumerate}
\item
$B(N) =
\oplus_{r\geq 0} B(N)_r$
is a differential graded bi-algebra over $\QQ$. (It follows that $B(N)$ is a differential graded Hopf algebra, since it
is a fact that antipode exists for a graded bi-algebra.)
\item
$B(N) =\oplus_{r\geq 0} B(N)_r$
is a direct sum decomposition to subcomplexes, and product, unit, co-product and counit are compatible with this decomposition.
\item
The product is graded-commutative with respect to the cohomological degree which we denote by $i$.
\end{enumerate}
With due caution one may say that $B(N)$ is an ``Adams graded'' differential graded Hopf algebra over $\QQ$, with graded-
commutative product; the first ``Adams grading'' refers to $r$, and the second grading refers to $i$, while
graded-commutativity of product is with respect to the grading $i$ (the product is neither graded-commutative or commutative with respect to $r$). We recall that graded Hopf algebra in the
literature means a graded Hopf algebra with graded-commutative product, so our $B(N)$ is a graded
Hopf algebra in this sense with respect to the grading $i$, but is not one with respect to the
``Adams grading'' $r$.
In the following
we will denote $H^0(B(N))$ by $\chi_N$.
The product, unit, coproduct, counit on $B(N)$ induce the corresponding
maps on ${\chi_N}$, hence ${\chi_N}$ is an ``Adams graded'' Hopf algebra over $\QQ$ in the following sense:
\begin{enumerate}
\item
${\chi_N}$ is a Hopf algebra over $\QQ$.
\item
With $({\chi_N})_r := H^0 (B(N)_r )$, one has 
a direct sum decomposition to subspaces  ${\chi_N}=\oplus_{r \geq 0} ({\chi_N})_r$;
the product, unit, coproduct and counit are compatible with this decomposition. 
\end{enumerate}

We also have the coproduct map 
$\Delta:H^0(B(\Q,N, M))\to \chi_N \otimes H^0(B(\Q,N, M))$
obtained from the homomorphism of complexes
$\Delta:B(\Q,N, M)\to B(N)\otimes B(\Q,N, M)$
given by
$$
\Delta([a_1|\cdots |a_s] m)
=
\sum_{i=0}^s\big([a_1|\cdots |a_i]\big)\otimes
\big([a_{i+1}|\cdots |a_s] m\big).
$$

We define the category of mixed Tate motives after Bloch-Kriz \cite{BK}. 
\begin{definition}[Adams graded $\chi_N$-co-modules, 
mixed Tate motives\index{mixed Tate motives}, \cite{BK}]

\begin{enumerate}
\item
Let 
$$V=\bigoplus_iV_i$$ be  a $\Q$-vector space of finite dimension with an Adams grading.
A linear map
$$
\Delta_V:V \to V\otimes {\chi_N}
$$
is called a  coaction of ${\chi_N}$
if the following conditions hold.
\begin{enumerate}
\item
$\Delta_V(V_i)\subset\oplus_{p+q=i}V_p\otimes ({\chi_N})_q$.
\item
(Coassociativity) 
The following diagram commutes.
$$
\begin{matrix}
V &\xrightarrow{\Delta_V} &V\otimes {\chi_N} \\
\Delta_V\downarrow \phantom{\Delta_V}& & \phantom{id_V\otimes \Delta_{{\chi_N}}}\downarrow 
{\rm id}_V\otimes \Delta_{{\chi_N}}
  \\
V\otimes {\chi_N} & \xrightarrow{\Delta_V\otimes {\rm id}_{{\chi_N}}} & V\otimes {\chi_N} \otimes {\chi_N}
\end{matrix}
$$
\item
(Counitarity)
The composite
$$
V \overset{\Delta_V}\to V\otimes {\chi_N} \xrightarrow{{\rm id}_V\otimes e} V,
$$
is the identity map,
where $e$ is the counit of ${\chi_N}$.
\end{enumerate}
An Adams graded vector space $V$ with a coaction $\Delta_V$ of ${\chi_N}$
 is called a right co-module
over ${\chi_N}$. For a right co-modules $V$, $W$ over ${\chi_N}$, a linear map
$V\to W$ is called a homomorphism of  co-modules over 
${\chi_N}$ if it preserves the Adams gradings 
and the coactions of ${\chi_N}$.
The category of  right co-modules over ${\chi_N}$ is denoted by 
co-${\rm rep}^f(\chi_N)$.  
\index{$(\operatorname{Com}^{gr}_{H^(B(N))})$}
\item
The category of mixed Tate motives 
${\rm MT_{BK}}$ 
\index{$({MT_{BK}}), ({MT_{BK}}_{F})$} 
over  $\Spec(F)$
is defined as the category 
co-${\rm rep}^f(\chi_N)$
of  right co-modules over ${\chi_N}$.  For an integer 
$r$, the Tate object
$\Q(r)$ is a  one-dimensional $\Q$-vector space of Adams degree $-r$, with trivial coaction
of $\chi_N$.
\end{enumerate}
\end{definition}
\begin{proposition}
\label{Kerd}
Let $V$ be an object of ${\rm MT_{BK}}$. Then the kernel of the map
\[\Delta_{V}\otimes \id- \id\otimes \Delta_{\chi_N}:\,\,
V\otimes \chi_N\to V\ot \chi_N\otimes \chi_N
\]
is equal to $\Delta_V(V)$.  Hence it is isomorphic to $V$ as a $\Q$-vector space by the counitarity
of $\Delta_V$.
\end{proposition}
\begin{proof}
By shifting the Adams grading of $V$ if necessary, we may assume that $V$ is of the form
 $\displaystyle \oplus_{i=0}^m V_i$. We prove the assertion by induction on $m$. When $m=
 0$, the coaction $\Delta_V$ is trivial since $\displaystyle \Delta_V(V)\subset  
 V\otimes(\chi_N)_0=V$, and the counitarity holds for $\Delta_V$. The assertion follows
 from the fact that the kernel of the map
 \[1\otimes \text{id}-\Delta_{\chi_N}:\,\,\chi_N\to \chi_N\otimes \chi_N\]
 is equal to $(\chi_N)_0=\Q$. Suppose that the assertion holds
 for $m$ and consider the case of $m+1$. Let $V_{\leq m}=\oplus_{i=0}^m V_i$. This is a sub comodule
 of $V$ and we have an obvious exact sequence
 \begin{equation}
 \label{basexact}
 0\to V_{\leq m}\to V\to V/ V_{\leq m}\to 0.
 \end{equation}
 Consider the commutative diagram with exact rows
  $$
\begin{matrix}
0&\rightarrow& V_{\leq m}\otimes \chi_N&\rightarrow &V\otimes \chi_N&\rightarrow& V/ V_{\leq m}\otimes \chi_N&\rightarrow& 0\\
&& D_1\downarrow \phantom{D_1}&&  D_2\downarrow\phantom{D_2} && D_3\downarrow\phantom{D_3} &&\\
0&\rightarrow& V_{\leq m}\otimes \chi_N\otimes \chi_N&\rightarrow &V\otimes \chi_N\otimes \chi_N&\rightarrow& V/ V_{\leq m}
\otimes \chi_N\otimes \chi_N&\rightarrow& 0.
\end{matrix}
$$  Here 
\[\begin{aligned}
D_1=&\Delta_{V}\otimes \id- \id\otimes \Delta_{\chi_N},  \\
D_2=&\Delta_{V_{\leq m}}\otimes \id- \id\otimes \Delta_{\chi_N}\,\,\, {\rm  and }\\
D_3=&\Delta_{V/V_{\leq m}}\otimes \id- \id\otimes \Delta_{\chi_N}.
\end{aligned}
\]  Let $C$ be the cokernel of $D_1$. We have an exact sequence
\[0\to \ker D_1\to \ker D_2\to \ker D_3\to C\to \cdots.\]
By the induction hypothesis $ \ker D_1=\Delta_{V_{\leq m}}(V_{\leq m})$ and $\ker D_3=
\Delta_{V/V_{\leq m}}(V/V_{\leq m})$.  By the exact sequence (\ref{basexact})  we have the equality
$\ker D_2=\Delta_V(V).$
 \end{proof}

\subsection{The category of flat connections}
 For the rest of this paper we assume that  $F$ is a number field. By Borel's results (\cite{Bo1},\, \cite{Bo2})
it follows that for a positive integer $r,$ we have
$$H^p_M(\Spec F, \Q(r))=0\,\, \text{ for } \,\,p\neq 1.$$
By (\ref{cohomN})  we have  $H^p(N^\bullet)=0$ for $p<0$,  and $H^0(N^\bullet)=\Q.$  So the cdga $N$
is cohomologically connected. In this case the category ${\rm MT_{BK}}$ is equivalent to another category
which is the one of flat $N$-connections. 
   
   Let $A$ be an Adams graded commutative dga (cdga) and let $M$ be an Adams graded $\Q$-vector space.
An $A$-connection for $M$ is a map of graded vector spaces
\[\Gamma:\quad M\to M\otimes A^1\] 
of Adams degree 0. This name is taken from \cite{Lev}. This is called a twisting matrix in \cite{KM}.  $\Gamma$ is said to be flat if $d\Ga+\Ga^2=0.$ Here
$d\Ga$ is the map
\[M\overset{\Ga}\to M\otimes A^1\overset{\text{id}\otimes d}{\to} M\otimes A^2\]
and $\Ga^2$ is
\[M\overset{\Ga}\to M\otimes A^1\overset{\Ga\otimes \text{id}}{\to} M\otimes A^1\otimes A^1\overset{\text{id}\otimes m}\to M\otimes A^2\]
where $m$ is the multiplication of $A$.  We denote by $\text{Conn}_A^0$ the category of flat $A$-connections, and
let $\text{Conn}_A^{0f}$ be its full subcategory of  flat connections on finite dimensional Adams graded $\Q$-vector spaces.

 Now we assume that $A$ is cohomologically connected. We denote by $A\{1\}$ the 1-minimal model of $A$. 
For the defintion of 1-minimal model see  \cite{KM} Part IV Definition 2.5. We have a homomorpshism
$A\{1\} \to A$ of cdgas, and this induces an isomorphism on $H^i$ for $i\leq 1$ and an injection on $H^2$. 

Set $QA:=A\{1\}^1$ and let $\part: QA\to \wedge ^2QA$ denote the differential $d:\,A\{1\}^1\to \wedge^2A\{1\}^1
=A\{1\}^2$. Then $(QA,\part)$ is a co-Lie algebra over $\Q$. Since $A$ is Adams graded, $QA$ is also Adams graded.
Let $\text{co-rep}(QA)$ be the category of Adams graded co-modules over $QA$, and let 
$\text{co-rep}^f(QA)$ be the category of finite dimensional Adams graded co-modules over $QA$. 
We denote by $\chi_A$ the Adams graded Hopf algebra $H^0(B(A))$.  Let 
$\text{co-rep}^f(\chi_A)$ denote the category of finite dimensional Adams graded co-modules over 
$\chi_A$. By  \cite{KM} Part IV Theorem 1.2 and Theorem 2.10 we have  the equivalence of filtered neutral Tannakian categories
\[\text{co-rep}^f(\chi_A)\sim \text{co-rep}^f(QA)\sim \text{Conn}^{0f}_{A\{1\}}.\]
Consider our case of $A=N.$  By \cite{KM} Part IV Corollary 1.4  we have an isomorphism
\[
\begin{aligned}
&\text{Ext}^1_{{\rm MT_{BK}}}(\Q, \Q(r))\simeq H^1((N\{1\})_r)\\
\simeq &H^1((N)_r)=CH^r(\Spec F, 2r-1)\\
=&H^1_M(\Spec F, \Q(r))=K_{2r-1}(F)\otimes\Q.
\end{aligned}
\] So the second part of the condition (B) holds. As for the first part of (B),  suppose there is a
flat $N$-connection $M$ which is an extension of $\Q(0)$ by $\Q(n)$ for $n\leq 0$.  
If we look at the construction  of a 1-minimal model $N\{1\}$ of $N$ which is explained for example
in the proof of  \cite{KM}  Part IV Theorem 3.7,  we see that $(N\{1\}^1)_r=0$ for $r\leq 0$.
Noting that the Adams degree of $\Q(n)$ is equal to $-n$, we see that  the connection
\[\Gamma:\quad M\to M\otimes N\{1\}^1\] 
must be trivial since $\Ga$ is of Adams degree 0.

\section{Hodge realization of ${\rm MT_{BK}}$}
\label{Hodger}
We give a description of our Hodge realization of ${\rm MT_{BK}}$. For the details we refer the reader to
\cite{part II}.  We will see that the conditions (C) and (D) hold for our Hodge realization.
\subsection{The complex $AC^\bullet$}
\label{sectionsemi-alg}
We give a brief description of the complex $AC^\bullet$ which is crucial for the construction
of our Hodge realization. For the detail the reader is refered to \cite{part II}.

Let $(z_1,\cdots,z_n)$ be the inhomogenious coordinates of $P^n=(\P^1_\CC)^n$. 
We set
\begin{align}
\label{notation of divisors}
&\bold H^n=\bigcup_{\substack{1\leq i\leq n\\
\alpha=0, \infty}} H_{i,\alpha}, \quad
\bold D^n=\bigcup_{i=1}^n \{z_i=1\},  \quad \square^n=P^n-\bold D^n.
\end{align}
\index{$\bold H^n$}
\index{$\bold D^n$}
\index{$\square^n$}
Let $K$ be a semi-algebraic triangulation of $P^n$ such that each cubical
face and  the divisor
$\bold D^n$ are subcomplexes of $K$.  
We denote by 
$$C_\bullet(K;\QQ)\,\,\text{ resp. }C_\bullet(K,\bold D^n;\QQ)
(=C_\bullet(K;\QQ)/C_\bullet(K\cap \bold D^n;\QQ))$$ 
the chain complex resp. the relative chain 
complex of $K$. The differential is denoted by $\delta$.
An element of $C_p(K;\QQ)$ for $p\geq 0$ is written as $\sum a_\si \si$
where the sum is taken over $p$-simplexes of $K$. By doing so, it is agreed upon that an orientation
has been chosen for each $\si$.  By abuse of notation, an element of $C_\bullet(K,\bold D^n;\QQ)$ 
is often described similarly. 

\begin{definition}
 For an element $\gamma=\sum a_\si \sigma$ of
$C_{p}(K;\QQ)$, we define
the support\index{support} $|\gamma|$ of $\gamma$ as the subset of $|K|$
 given by
\begin{equation}
\label{support}
|\gamma|=\bigcup_{
\substack{\sigma\in K_p\\ a_\sigma\neq 0}} \sigma.
\end{equation}
\end{definition}

\begin{definition}
\label{def:semi-alg current}
Let $p \geq 0$ be an integer.
\begin{enumerate}

\item(Admissibility) A semi-algebraic subset $S$ of $P^n$
is 
said to be admissible\index{admissible} if for each cubical face $H$, 
the inequality 
$$
\dim(S\cap (H-\bold D^n))\leq \dim S -2\,\codim H
$$
holds.  Here note that $\dim(S\cap (H-\bold D^n))$ and $\dim S$ are the dimensions as semi-algebraic sets,
and $\codim H$ means the codimension of the subvariety $H$ of $P^n$.

\item Let $\gamma$ be an element of $C_p(K,\bold D^n; \QQ)$.
Then $\ga$ is 
said to be admissible if 
the support of a representative of $\ga$ in $C_p(K;\QQ)$ is admissible. This condition is independent of the choice
of a representative. 
\item We set 
$$
AC_{p}(K,\bold D^n;\QQ)
=\{\gamma\in C_{p}(K,\bold D^n;\QQ))
\mid 
\gamma \text{ and } \delta \gamma \text{ are admissible }\}
$$
\index{$AC_{p}(K,\bold D^n;\QQ^{q})$}
\end{enumerate}
\end{definition}
Two semi-algebraic triangulations $K$ and $K'$ have a common subdivision. Thus we can 
take the inductive limit of the complexes $C_\bullet$ and $AC_\bullet$  over subdivisions.

\begin{definition}

\label{def:definition of AC with inductive limit}
We set
\begin{align*}
C_\bullet(P^n,\bold D^n;\QQ) =
\underset{\underset{K}\longrightarrow}{\lim } \ 
C_\bullet(K,\bold D^n;\QQ), \quad
AC_\bullet(P^n,\bold D^n;\QQ) =
\underset{\underset{K}\longrightarrow}{\lim } \ 
AC_\bullet(K,\bold D^n;\QQ).
\end{align*}
Here the limit is taken on the directed set of  triangulations.

\end{definition}
\index{$C_\bullet(P^n,\bold D;\QQ^\bullet)$}
\index{$AC_\bullet(P^n,\bold D;\QQ^\bullet)$}


\noindent A proof of the following Proposition is given in \cite{part II}.

\begin{proposition}[Proposition 2.8 \cite{part II}] 
\label{prop: moving lemma}
The inclusion of complexes
\begin{equation}
\label{moving quasi-iso}
\iota:\,\,AC_\bullet(P^n,\bold D; \QQ) 
\to C_\bullet(P^n,\bold D; \QQ)
\end{equation}
is a quasi-isomorphism.
\end{proposition}

For a codimension one face $H_{i,\al}$ we can define the intersection map
$$
\part_{i,\al}:
AC_\bullet(P^n,\bold D^n;\QQ)\to
AC_{\bullet-2}(H_{i,\al},\bold D^n\cap H_{i,\al};\QQ)
$$
which commutes with the differential $\delta$.  We call $\part_{i,\al}$ a {\it face map.} The construction of face maps
is given in \cite{part II} Section 3. Face maps commute with each other,
and we can define the {\it cubical differential}

\begin{equation}
\label{total face map definition} 
\part =\sum_{i=1}^n (-1)^{i-1}
(\part_{i,0}-\part_{i,\infty}):\,\,
AC_\bullet(P^n,\bold D^n;\QQ)\to
AC_{\bullet-2}(P^{n-1},\bold D^{n-1};\QQ).
\end{equation}

\subsection{Cauchy-Stokes formula}

Let $K$ be a  triangulation of $P^n$. 
We define the rational differential form $\omega_n$ on $P^n$ by
$$\omega_n =\frac{1}{(2\pi i)^n}\frac{dz_1}{z_1}\wedge
\cdots \we\frac{dz_n}{z_n}.
$$ 
\index{$\omega_n$}
By applying  the Main Theorem 4.4 of \cite{part I} to the case where $A=\si$, $m=\dim \si=n$ and $\omega=\omega_n$
we have the  following theorem.
 \begin{theorem}[Theorem 4.1 \cite{part II}] 
\label{convadmiss}
Let $\si$ be an admissible $n$-simplex of $K$.  
The integral 
\begin{equation}
\label{convergenct integral1}
\int_\sigma \omega_n
\end{equation}
converges absolutely.

\end{theorem}

\begin{definition}
Let $\gamma$  
be an element in $AC_n(K,\bold D^n;\QQ)$, and let $\sum_{\sigma}a_\si \sigma$
be a representative of $\gamma$ in $C_n(K;\QQ)$.
We define $I_n(\gamma)\in \CC$
by\index{$I_n$}
\begin{equation}
\label{def: def of I}
I_n(\gamma)=(-1)^{\frac{n(n-1)}2}\sum_{\sigma}\int_{\sigma}a_\sigma\omega_n
\end{equation}
where the sum is taken over the $n$-simplexes $\si$ not contained in $\bold D^n$.
The right hand side of {\rm (\ref{def: def of I})} does not depend on the choice of a representative of $\gamma$.
The map $I_n$ is compatible with subdivisions of triangulations, and we obtain a map
$$
I_n:AC_n(P^n,\bold D^n;\QQ)\to \CC.
$$
\end{definition}
In \cite{part II} Section 4  the following result is proved. (Theorem 4.3 \cite{part II}) 
\begin{theorem}[Cauchy-Stokes formula]
\label{th:generalized Cauchy formula}
\index{generalized Cauchy formula}
Let 
$\ga$ be an element in
 $AC_{n+1}(P^n,\bold D^n;\QQ)$. 
Then we have the equality
\begin{equation}
 \label{eq:cauchy formula} 
I_{n-1}(\part \gamma)+(-1)^nI_n(\delta \gamma)=0.
\end{equation}
\end{theorem}


\begin{definition}[The complex $AC^\bullet$] 
\label{def:admissible triple complex}
For an integer $k$, let 
$$
 AC^k=\underset{n-i=k}\bigoplus
AC_i(P^n, \bold D^n; \QQ)^{\rm alt}
$$
Here we set 
$$
AC_i(P^0, \bold D^0; \QQ)^{\rm alt}=
\left\{
\begin{array}{cc}
\QQ & i=0\\
0& i\neq 0
\end{array}
\right.
$$
We define the  differential $d$ of $AC^\bullet$ 
by the equality
\begin{equation}
\label{eq:definition of differential D}
d(\alpha)=\partial \alpha+(-1)^n\delta\alpha \text{ for } \al\in AC_i(P^n,\bold D^n;\QQ)^{\rm alt}.
\end{equation}
Here we use the fact that the maps $\part$ and $\delta$ commute with the projector $\rm Alt$. 
cf. Lemma 4.3 \cite{BK}.

\end{definition}
We have a natural inclusion $Z^r(\Spec {F}, 2r-i)\to AC_{2r-2i}(P^{2r-i},\bold D^{2r-i}; \QQ)$
which induces an inclusion $\iota:\, N\to AC^\bullet$.  The following result is proved in
\cite{part II} Section 5. 
\begin{proposition}[Proposition 5.2 \cite{part II}]
\label{properties of tc}
\begin{enumerate}
\item Let $u:\,\QQ\to AC^\bullet$ be the inclusion defined by identifying
$\QQ$ with $AC_0(P^0,\bold D^0;\QQ)^{\rm alt}$. Then the map $u$ is a quasi-isomorphism.

\item The inclusion $\iota$ and the product defined by
$$
\begin{array}{c}
N_r^i\otimes AC_j(P^n,\bold D^n;\QQ)^{\rm alt}\to AC_{2r-2i+j}(P^{n+2r-i},\bold D^{n+2r-i};\QQ)^{\rm alt}\\
z\times \ga\mapsto z\cdot \ga:={\rm Alt}(\iota(z)\times \ga)
\end{array}
$$
makes $AC^\bullet$ a differential graded $N$-module i.e. the product
sends $N^i_r\otimes AC^j$ to $AC^{i+j}$, and one has the derivation formula; for $z\in N_r^i$ and $\ga\in AC^\bullet$,
one has 
$$ d(z\cdot \ga)=\part z\cdot \ga+(-1)^iz\cdot d\ga$$

\item The map $I$ defined by 
$I=\sum_{n\geq 0}I_n:\,{AC}^\bullet\to \CC$ is a map of complexes. Here the map $I$
on $AC_0(P^0, D^0; \QQ)^{\rm alt}=\QQ$ is defined to be the natural inclusion of $\QQ$ to $\CC$. The map
$I$ is a homomorphism of $N$-modules; for $z\in N$ and $\ga\in AC^\bullet$,
we have an equality
$$I(z\cdot \ga)=\e(z)I(\ga)$$

\item Let $I_\CC: AC^\bullet \otimes \CC \to \CC$ be the map defined by the composition
\newline $\,AC^\bullet\ot \CC \overset {I\ot {\rm id}}\to \CC\ot \CC\overset{m}\to \CC $. Here the map $m$ is the multiplication.
Then  this map $I_\CC$ is a quasi-isomorphism.
\end{enumerate}

\end{proposition}

\subsection{A Hodge realization of ${\rm MT_{BK}}$.}
\label{Hhodge}

We recall the definition of the Tate Hodge structure.  For an integer $r$, let
$\Q(r)=\Q(2\pi i)^r$\index{$\Q(r)$} with the weight filtration $W$ defined by
$\Q(r)=W_{-2r}\supset W_{-2r-1}=0$, and let $\C(r)=\C$ with
the Hodge filtration $F$ defined by 
$\C(r)=F^{-r}\supset F^{-r+1}=0$.\index{$\C(r)$}
We define the mixed Tate Hodge structure $\Q_{Hg}(r)$\index{$\Q_{Hg}(r)$} 
of weight $-2r$ by
the $\QQ$-mixed Hodge structure $(\QQ(r),\CC(r),F,W)$ where the 
comparison map $c:\,\Q(r)\to \C(r)$ is defined by the inclusion map.
This is a Hodge structure of type $(-r,-r)$.
For  a mixed Hodge structure $H_{Hg}$,
$H_{Hg}\otimes \QQ_{Hg}(r)$ is denoted by
 $H_{Hg}(r)$.
\index{$H(r)$, $H_{Hg}(r)$}  A (finite dimensional) mixed Hodge structure is called a mixed Tate
Hodge structure if the weight graded quotients
are isomorphic to direct sums of Tate Hodge structures.

In this section, we define a possibly infinite dimensional mixed Hodge structure $\cal H_{Hg}$.
Let $\si$ be a complex embedding of $F$.
We define the bar complexes $\cal B_b$ and $\cal B_{dR}$ by 
$$
\cal B_b=B(N,AC^\bullet) \,\text{ and }\, \cal B_{dR}=B(N,F).
$$

We introduce the weight filtration $W_\bullet$ on 
$\cal B_b$ and
$\cal B_{dR}$ as follows.
\index{$W_n\cal B_b, W_n\cal B_{dR}$} 
$$
 W_n\cal B_b= 
\bigoplus_{2(r_1+\cdots +r_s)
\leq n,s\geq 0,\,r_i>0}N_{r_1}\otimes \cdots \otimes N_{r_s}\otimes AC^\bullet.
$$
$$
 W_n\cal B_{dR}= 
\bigoplus_{2(r_1+\cdots +r_s)\leq n,s\geq 0,\,r_i>0}N_{r_1}\otimes
\cdots \otimes N_{r_s} \otimes F.
$$
The
Hodge filtration $F^\bullet$ on
$\cal B_{dR}$ is defined by 
$$
F^p\cal B_{dR}=
\bigoplus_{r_1+\cdots +r_s\geq p,s\geq 0,\,r_i>0}
N_{r_1}\otimes
\cdots \otimes N_{r_s} \otimes F.
$$
We have a canonical isomorphism 
$$\text{gr}^W_{2r}\cal B_b=
\bigoplus_{r_1+\cdots +r_s=r,\,s\geq 0,\,r_i>0}N_{r_1}\otimes \cdots \otimes N_{r_s}\otimes AC^\bullet
$$
resp.
$$\text{gr}^W_{2r}{\cal B}_{dR}=
\bigoplus_{r_1+\cdots +r_s=r,\,s\geq 0,\,r_i>0}N_{r_1}\otimes \cdots \otimes N_{r_s}\otimes
F.$$

We define the comparison map $c:\,\,\cal B_b\to
{\cal B}_{dR}\otimes_{F}\CC $ to be $\text{id}\otimes (2\pi i)^{-r} I$
on $\text{gr}^W_{2r}\cal B_b$. By Proposition \ref{properties of tc} (3), the comparison map $c$ induces a homomorphism of complexes
from $\cal B_b$ to ${\cal B}_{dR}\otimes_{F}\CC$. 

\begin{definition}
\label{def of universal mixed Hodge structure}
 We define the Betti part $\cal H_b$
and  the de Rham part $\cal H_{dR}$
of $\cal H_{Hg}$
by $\cal H_b=H^0(\cal B_b)$ and 
$\cal H_{dR}=H^0(\cal B_{dR})$.
\index{$\cal H_{B}, \cal H_{dR}$}

\end{definition}
By 
Proposition \ref{properties of tc} (3), 
the map $c$ induces  a homomorphism of complexes
\index{$F^p\cal B_{dR}$}
\begin{equation}
\label{weight compatible}
c:\, W_n\cal B_b\ot\CC \to W_n\cal B_{dR}\otimes_{F}\CC.
\end{equation}
Here we regard $F$ as a subfield of $\CC$ by $\si$. The weight and Hodge filtrations on 
$\cal B_b$ and 
$\cal B_{dR}$
induce those of 
$\cal H_b$ and
$\cal H_{dR}$, respectively.  For $\cal H_{dR}$  these filtrations are simple.  We have
\[W_n \cal H_{dR}=\oplus_{2r\leq n} (\chi_N)_r\otimes {F}, \text{ and }
F^p \cal H_{dR}=\oplus_{r\geq  p} (\chi_N)_r\otimes {F}.\]

\begin{proposition}[Proposition-Definition 5.5 \cite{part II}]
\label{hodge tate}
\begin{enumerate}
\item The map $c$ induces a filtered quasi-isomorphism
$\cal B_b\otimes \CC\to \cal B_{dR}\otimes \C$ with respect to $W$.
\item
\label{gr of r is degree r part}
We have a canonical isomorphism of vector spaces
$$
Gr^W_{2r}\cal H_b\to (\chi_N)_r
$$
\item
The isomorphism $c:\cal H_b\otimes \CC\to \cal H_{dR}\otimes \C$ induces an isomorphism of weight filtered
 vector spaces from $\cal H_b\otimes \C$ to $\cal H_{dR}\otimes \C$.

\end{enumerate}

\end{proposition}
\begin{proof} 
(1)  One sees that the quotient $Gr_{2r}^W\cal B_B$  is the tensor product
\[ 
B(N)_r\ot AC^\bullet
\]
as a complex. By Proposition \ref{properties of tc} (4) the map $c$ induces a quasi-isomorphism
\newline $Gr_{2r}^W\cal B_B\ot \CC\to Gr_{2r}^W\cal B_{dR}\otimes \C$.
 
(2)
 We consider the spectral sequences for the filtration $W$:
 \begin{align*}
E_1^{p,q}=H^{p+q}(Gr^W_{-p}\cal B_B) &\Rightarrow E^{p+q}=
 H^{p+q}(\cal B_B) \\
\ '  E_1^{p,q}=H^{p+q}(Gr^W_{-p}\cal B_{dR}\otimes_F\CC) &\Rightarrow \ 'E^{p+q}=
 H^{p+q}(\cal B_{dR}\otimes_F\CC). 
 \end{align*}
Since the morphism of complexes $c:\,\cal B_B\otimes \CC\to \cal B_{dR}\otimes_F\CC$
 is a filtered quasi-isomorphism, the map
 
  $$E_1^{p,q}\otimes \CC \to \ 'E_1^{p,q}$$
induced by $c$  is an isomorphism. Note that we have 
$$E_1^{p,q}\ot \C=H^{p+q}(Gr^W_{-p}({\cal B}_B\ot \C))$$ since tensoring with $\C$ is an exact functor.
Since the complex $\cal B_{dR}$ is  the direct sum 
$\dis\underset{r}\oplus B(N)_r\ot F$, 
the spectral sequence $\ 'E_*^{*,*}$ degenerates at $E_1$-term, and
as a consequence  $E_*^{*,*}$ also degenerates at $E_1$-term.
 Therefore the vector space
$Gr_{2r}^WH^0(\cal B_B)$ is canonically isomorphic to $H^0(Gr_{2r}^W\cal B_B)$.
By Proposition \ref{properties of tc} (1) we have
$$
H^0(Gr^W_{2r}\cal B_B)
={\chi_N}_r\ot H^0(AC^\bullet)={\chi_N}_r\otimes \Q.
$$
The assertion (3)  also follows from this argument.
\end{proof}


Let $V=\oplus_r V_r$ be an Adams graded $\Q$-vector space of finite dimension.
 The mixed Tate 
Hodge structure associated to $V$ which is denoted by
$V_{Hg}$, is defined as follows. $V_{Hg}$ is the triple $(V_b, V_{dR},c)$.
$V_b=V$ as a $\Q$-vector space, and the weight filtration is defined
so that $W_nV_b=\oplus_{2r\leq n}V_r$. $V_{dR}=V\otimes \C$ as a vector space,
and  $W_nV_{dR}=\oplus_{2r\leq n}V_r\ot \C$. The Hodge filtration
is defined by $F^pV_{dR}=\oplus_{r\geq p} V_r\ot \C$. We have a canonical isomorphism
$\text{gr}^W_{2r} V_b=V_r$. The comparison map $c$ is defined on 
$\text{gr}^W_{2r} V_b=V_r$ to be the multiplication by $(2\pi i)^{-r}$.  
Let $V$ be an Adams graded ${\chi_N}$-co-module.
The coaction $\Delta_V:\,V\to V\ot {\chi_N}$ induces a map of
mixed Tate Hodge structures
$$\Delta_{V_{Hg}}:\,\,V_{Hg}\to V_{Hg}\ot ({\chi_N})_{Hg}.$$
Similarly the coproduct 
$$\Delta_b:\,{\cal H}_b\to \chi_N
\ot\cal H_b$$ 
and
$$\Delta_{dR}:\,{\cal H}_{dR}\to \chi_N
\ot {\cal H}_{dR}$$ 
induce a map of mixed Tate Hodge structures
$$\Delta_{Hg}:\,\,{\cal H}_{Hg}\to ({\chi_N})_{Hg}\ot {\cal H}_{Hg}.$$

\begin{proposition-definition}[Realization functor]
\label{defphi}
Let $V$ be an object of ${\rm MT_{BK}}$. Let

\[
\Phi(V)_b=\ker(\Delta_V\otimes \id- \id\otimes \Delta_{{\cal H}_b}:\,\,
V\otimes\cal H_b\to V\ot \chi_N \ot\cal H_b)\]
and
\[
\Phi(V)_{dR}=\ker(\Delta_V\otimes \id- \id\otimes \Delta_{{\cal H}_{dR}}:\,\,
V\otimes {\cal H}_{dR}\to V\ot \chi_N \ot {\cal H}_{dR}).\]
Then there is an isomorphism of weight filtered vector spaces
\[c:  \Phi(V)_b\otimes \C\to \Phi(V)_{dR}\]
 induced by the isomorphism
$c:\cal H_b\otimes \C\to {\cal H}_{dR}\otimes \C$  given in {\rm Proposition \ref{hodge tate} (3)}.
The triple $\Phi(V):=(\Phi(V)_b,  \Phi(V)_{dR}, c)$ is a mixed Tate Hodge structure. $\Phi(V)$ 
is defined to be the Hodge realization of $V$.
\end{proposition-definition}
\begin{proof}
By Proposition \ref{hodge tate} (3) the map $c$ induces an isomorphism of weight filtered
vector spaces $c:  \Phi(V)_b\otimes \C\to \Phi(V)_{dR}\otimes\C.$ By Proposition \ref{Kerd}
$\Phi(V)_{dR}\otimes\C$ is isomorphic to $\oplus V_i\otimes \C$  with the weight and Hodge filtrations
defined in terms of Adams grading.  Hence $\Phi(V)$ is a mixed Tate Hodge structure.
\end{proof}


\subsection{Hodge realization of the category of flat connections}
\label{flat conn}
We will show that the Hodge realization functor $\Phi$ given in Definition \ref{defphi} is equivalent to a simpler functor.
As we have said the category $\rm MR_{BK}$ is equivalent to the category of flat connections
 $\text{Conn}^{0f}_{N\{1\}}$.
We will show that there is a natural Hodge realization of $\text{Conn}^{0f}_{N\{1\}}$,
which is compatible with the functor $\Phi$ defined  in Definition \ref{defphi} under the equivalence of the categories
${\rm MT_{BK}}=\text{co-rep}^f(\chi_N)\sim \text{Conn}^{0f}_{N\{1\}}$. Let $M=\oplus_r M_r$ be a finite dimensional Adams graded
$\Q$-vector space with a flat connection $\Ga:\,\,M\to M\otimes N\{1\}^1$. The Hodge complex associated to
$M$ is the triple $(M_b,\,M_{dR},\,c)$ defined as follows. $M_b=M\otimes AC^\bullet$ as a bigraded $\Q$-vector space.
Here  $AC^\bullet$ is the complex of semi-algebraic chains defined in Definition \ref{def:admissible triple complex}. The differential of $M_b$
is the sum $d_1+d_2$ where $d_1=1\otimes d_{AC^\bullet}$ and $d_2$ is the composite $M\otimes AC^\bullet
\overset{\Ga\otimes \text{id}}\to M\otimes N\{1\}^1\otimes AC^\bullet \overset{\text{id}\otimes m} \to M\otimes AC^\bullet$.
Here $m$ is the multiplication $ N\{1\}^1\otimes AC^\bullet\to AC^\bullet$. The weight filtration
of $M_b$ is defined by 
$$W_nM_b=\bigoplus_{2r\leq n}M_r\otimes AC^\bullet.$$
The complex $M_{dR}:=M\otimes F$ concentrated in cohomological degree 0 with trivial differential. 
The weight filtration of $M_{dR}$ is defined by
$$W_nM_{dR}=\bigoplus_{2r\leq n}M_r\otimes F$$
and the Hodge filtration of $M_{dR}$
is defined by $$F^pM_{dR}=\bigoplus_{p\leq r} M_r\otimes F.$$ 
We have a canonical isomorphism 
$$\text{gr}^W_{2r}(M_b)\simeq M_r\otimes AC^\bullet$$
and
$$\text{gr}^W_{2r}(M_{dR})\simeq M_r\otimes F.$$
The comparison map $c:\,\,M_b\to M_{dR}\otimes_{F}\CC$ is defined to be $\text{id}\otimes(2\pi i)^{-r} I$ on
$M_r\ot AC^\bullet$ where $I:\,AC^\bullet \to \C$
is the map defined in Proposition \ref{properties of tc} (3).

The proof of the following proposition is similar to that of Proposition \ref {hodge tate}.
\begin{proposition}
\label{hodgeconn}
The triple $(H^0(M_b),\,H^0(M_{dR})(=M_{dR}),\,c)$ is a mixed Tate Hodge structure. 
\end{proposition}

\begin{definition}
\label{defpsi}
The Hodge realization $\Psi(M)$ of $M$ is defined to be the triple
\newline $(H^0(M_b),\,H^0(M_{dR})(=M_{dR}),\,c)$. 
\end{definition}

We will show that the  the Hodge realizations $\Psi$ and $\Phi$ 
defined in definition \ref{defphi}  are equivalent.
 By iterating the map $\Ga$ we obtain a map $\Ga_n:\,M\to M\otimes (N\{1\}^1)^{\otimes n}$
for each $n\geq 0$. Taking the sum $\sum_{n\geq 0}\Ga_n$, we obtain a map
$\Delta_M:\,\, M
\to M\otimes T(N\{1\}^1)$. Here $T(N\{1\}^1)=\oplus_{s\geq 0}(N\{1\}^1)^{\ot s}$. 
Since the connection $\Ga$ is flat, $\Delta_M$ is actually a map
$M\to M\otimes H^0(B(N\{1\}))$. The maps $\Delta_M$ induces an equivalence of the categories
\[ \text{Conn}^{0f}_{N\{1\}}\to \text{co-rep}^f(\chi_A).\]
Let $\Delta_{M_b}$ be the map
\[M\otimes AC^\bullet\overset{\Delta_M\otimes 1}\to M\otimes T(N\{1\}^1)\otimes AC^\bullet.\]
It follows from the definition of the bar complex $B(N,AC^\bullet)$ that $\Delta_{M_b}$ defines a map
of complexes
\[\Delta_{M_b}:\,\,M_b\to M\otimes B(N,AC^\bullet)=M\otimes \cal B_b.\]
Here the $\Q$-vector space $M$ on the target is a trivial complex.
Similarly, we have the map of complexes
\[\Delta_{M_{dR}}:\,\,M_{dR}\to M\otimes B(N,F)=M_{dR}\otimes \cal B_{dR}.\]

\begin{theorem}
\label{comparison}
The pair $(\Delta_{M_b},\Delta_{M_{dR}})$ induces a map of mixed Hodge 
structures from $\Psi(M)$ to $\Phi(M)$ 
which is an isomorphism.
\end{theorem}
\begin{proof} Since $M$ is a co-module over ${\chi_N}$, $\Delta_{M_b}(M_b)$
is contained in the kernel of the map 
$$\Delta_M\otimes 1-1\otimes \Delta_b:\,\,
M\otimes B(N,AC^\bullet)\to M\otimes B(N)\otimes B(N,AC^\bullet).$$
Hence $\Delta_{M_b}$ induces a map from $H^0(M_b)$ to
$\Phi(M)_b$.  We need to show that it is an isomorphism of weight filtered
vector spaces. 
By Proposition \ref{hodge tate}  (3), we have an isomorphism of weight 
filtered vector spaces
$$M\otimes H^0(B(N,AC^\bullet))\otimes \C
\to M\otimes H^0(B(N,\C))$$
resp. 
$$M\otimes {\chi_N}\otimes H^0(B(N,AC^\bullet))\otimes \C
\to M\otimes {\chi_N}\otimes H^0(B(N,\C)).$$
So our assertion is follows from the fact that the map
induced by $\Delta_{M_{dR}}$

$$H^0(M_{dR}\otimes\C)=M\otimes \C
\to \text{Ker}(\Delta_M\otimes 1-1\otimes \Delta:\,\,
M\otimes H^0(B(N,\C))\to M\otimes {\chi_N}\otimes H^0(B(N,\C)))$$
is an isomorphism of weight filtered vector spaces. This is true by Proposition \ref{Kerd}.

\end{proof}

\begin{theorem}
\label{tensor}
The functor $\Psi$ is compatible with the tensor product.
\end{theorem}
\begin{proof}
We start with defining a product $\circ$ on $AC^\bullet.$
For $a\in AC_{d_a}(P^{n_a}, \bold D^{n_a}; \QQ)^{\rm alt}$
and $b\in AC_{d_b}(P^{n_b}, \bold D^{n_b}; \QQ)^{\rm alt}$,
let $a\circ b$ be $(-1)^{d_an_b}a\cdot b.$ Then we have the following.
\begin{proposition}
\label{bullet}
\begin{enumerate}
\item The product $\circ$ induces a map of complexes
\[AC^\bullet\otimes AC^\bullet\to AC^\bullet\]
which is a quasi-isomorphism.

\item The product $\circ$ is associative.

\item The graded commutativity of $\circ$ i.e. the equality
\[b\circ a=(-1)^{\deg a\deg b}a\circ b\]
holds if $d_a$ or $d_b$ is even.

\item We have the equality
\[I(a) I(b)=I(a\circ b).\]
\end{enumerate}
\end{proposition}

\begin{proof}
The proof of this proposition is rather straightforward. The facts we use in the proof are as follows.
Let $a\in AC_{d_a}(P^{n_a}, \bold D^{n_a}; \QQ)^{\rm alt}$ and
$b\in AC_{d_b}(P^{n_b}, \bold D^{n_b}; \QQ)^{\rm alt}$.

\begin{itemize}
\item  The equality 
$da=\part a+(-1)^{n_a}\delta a$ holds by definition.
\item The equality 
$$b\cdot a=(-1)^{n_an_b}a\cdot b$$
holds.
\item The equality
$$\part(a\cdot b)=\part a\cdot b+(-1)^{n_a}a\cdot \part b$$
holds.

\item  The equality  
$$\delta(a\cdot b)=\delta a\cdot b+(-1)^{d_a}a\cdot \delta b$$
holds.
\item The product $\circ$ induces an isomorphism
$$AC_0(P^0, \bold D^0; \QQ)\otimes AC_0(P^0, \bold D^0; \QQ)\to AC_0(P^0, \bold D^0; \QQ).$$

\item The integral $\dis \int_a\omega_{n_a}=0$ unless $d_a=n_a$.
\end{itemize}
\end{proof}
Let $M_1$ and $M_2$ be objects of ${\rm MT_{BK}}$ regarded as objects
of $\text{Conn}^{0f}_{N\{1\}}$. Let $T$ be the map 
\[(M_1)_b\otimes (M_2)_b\to (M_1\otimes M_2)_b\]
defined by 
\[T((m_1\otimes \ga_1)\otimes (m_2\otimes\ga_2))=m_1\otimes m_2\otimes (\ga_1\circ \ga_2).\]
\begin{lemma}
\label{differential}
$T$ is a map of complexes.
\end{lemma}
\begin{proof}
The differential of $(M_1)_b\otimes (M_2)_b$ is defined by 
\[\begin{aligned}
&d((m_1\otimes \ga_1)\otimes (m_2\otimes\ga_2))\\
=&d(m_1\otimes \ga_1)\otimes (m_2\otimes\ga_2)+(-1)^{\deg \ga_1}
(m_1\otimes \ga_1)\otimes d(m_2\otimes\ga_2).
\end{aligned}
\]
The differential $d=d_1+d_2=\text{id}\otimes d_{AC^\bullet}+m\circ (\Gamma\otimes \text{id}).$
For $d_1$ the compatibility with $T$ follows from Proposition \ref{bullet} (1).
We consider the case of $d_2$. This is a matter of sign, and we may assume
that $\Gamma(m_1)=m_1'\otimes n_1$ and  $\Gamma(m_2)=m_2'\otimes n_2$
for simplicity. Then we have
\[\begin{aligned}
&d_2((m_1\otimes \ga_1)\otimes (m_2\otimes\ga_2))\\
=&(m_1'\otimes n_1\circ \ga_1)\otimes (m_2\otimes\ga_2)+(-1)^{\deg \ga_1}
(m_1\otimes \ga_1)\otimes (m_2'\otimes n_2\circ \ga_2).
\end{aligned}
\]
So that we have 
\[\begin{aligned}
&T\circ d_2((m_1\otimes \ga_1)\otimes (m_2\otimes\ga_2))\\
=&(m_1'\otimes m_2)\otimes (n_1\circ \ga_1\circ \ga_2)+(-1)^{\deg \ga_1}
(m_1\otimes m_2')\otimes (\ga_1\circ n_2\circ \ga_2).
\end{aligned}
\]
On the other hand we have
\[\begin{aligned}
&d_2\circ T((m_1\otimes \ga_1)\otimes (m_2\otimes\ga_2))\\
=&d_2(m_1\otimes m_2)\otimes(\ga_1\circ \ga_2)\\
=& \{(m_1'\otimes n_1)\otimes m_2+m_1\otimes (m_2'\otimes n_2)\}
\otimes (\ga_1\circ \ga_2))\\
=& m_1'\otimes m_2\otimes( n_1\circ\ga_1\circ \ga_2)+m_1\otimes m_2'
\otimes( n_2\circ\ga_1\circ\ga_2).
\end{aligned}
\]
By Proposition \ref{bullet} (3) we have
\[n_2\circ \ga_1=(-1)^{\deg n_2\deg \ga_1}\ga_1\circ n_2.\]
Since $\deg n_2=1$ our assertion is proved.
\end{proof}
By Lemma \ref{differential} the map $T$ induces a map of weight filtered complexes
\[(M_1)_b\otimes (M_2)_b\to (M_1\otimes M_2)_b.\]
We look at the map induced by $T$ on $E_1^{p,q}$ of weight spectral sequences. This is equal to the map
\[E_1^{p,q}=H^{p+q}(\text{gr}^W_{-p}((M_1\otimes M_2)\otimes AC^\bullet\otimes AC^\bullet)
\longrightarrow H^{p+q}(\text{gr}^W_{-p}(M_1\otimes M_2)\otimes AC^\bullet)\]
induced by the product $\circ$, so this is an isomorphism by Proposition
\ref{bullet} (1).  Hence the map $T$ induces an isomorphism
of  weight filtered vector spaces
\[H^0((M_1)_b\otimes (M_2)_b)\to H^0((M_1\otimes M_2)_b).\] 
We check the compatibility of $T$ and the comparison map. We need the following diagram commute.
\[
\begin{matrix}
(M_1)_b\otimes (M_2)_b&\overset{I\otimes I} \rightarrow &(M_1)_{dR}\otimes (M_2)_{dR}\otimes \C\\
T \downarrow \phantom{T}&&\text{id}\downarrow \phantom{\text{id}}\\
(M_1\otimes M_2)_b&\overset{ I}\rightarrow& (M_1\otimes M_2)_{dR}\otimes \C.
\end{matrix}
\]
This holds by Proposition \ref{bullet} (4).  
\end{proof}
Let $\si$ be an embedding of $F$ into $\C$. $\si$ defines a Hodge realization of ${\rm MT_{BK}}$.
$\si$ also defines a regulator map from  $K_{2r-1}(F)$ to $\text{Ext}^1_{MHS} (\Q(0),\,\Q(r))=\C/(2\pi i)^r\Q.$
\begin{proposition}
The following diagram commutes.
\[
\xymatrix{
K_{2r-1}(F)\otimes \Q \ar[rr]\ar[rd]^{\text{reg}}&& Ext^1_{\rm MT_{BK}}(\Q(0),\,\Q(r))\ar[ld]^{\text{real}}\\
&\text{Ext}^1_{MHS} (\Q(0),\,\Q(r))&
}
\]

\end{proposition}
\begin{proof}
Let $z$ be a non-zero element of $K_{2r-1}(F)\otimes \Q=CH^r(\Spec F, 2r-1)$ and $Z\in Z^r(\Spec F, 2r-1)
^{\alt}$
be a higher Chow cycle which represents $z$. The object $M\in {\rm MT_{BK}}$ which is an extension of
$\Q(0)$ by $\Q(r)$ defined by $z$ is described as follows. The cycle $Z$ defines an element $[Z]$
in $(\chi_N)_r.$ As a $\Q$-vector space
$$M=M_0\oplus M_{-r}$$ where $M_0$ resp. $M_{-r}$ is a one dimensional $\Q$-vector space of
Adams degree 0 resp. $-r$.  Let $e_0$ resp. $e_{-r}$
be a generator of $M_0$ resp. $M_{-r}$. Then the coaction of $\chi_N$ on $M$ is defined 
by the equalities 
$$\Delta e_0=e_0+e_{-r}\otimes [Z],\quad \Delta e_{-r}=e_{-r}.$$
The definition of $\Delta$ is independent of the choice of the generators up to isomorphisms.
If we look at the proof of \cite{KM} Part IV Theorem 3.7, we see that there exists a 1-minimal
model $N\{1\}$ of $N$ whose image in $N$ contains $Z$. 
As an object of $\text{Conn}^{0f}_{N\{1\}}$  the connection $\Ga$ of $M$ is defined by the equalities
\[ \Ga(e_0)=e_{-r}\otimes Z,\,\,\text{ and } \,\,\Ga(e_{-r})=0.\]
We compute the image $\Psi(M)$ under the functor $\Psi$.
By definition we have $M_b=M\otimes AC^\bullet$ and $M_{dR}=M\otimes F$. 

We show that there is an element
$G\in AC^0$ such that $dG=-Z$. $Z$ defines a homology class in
$$H_{2r-2}(P^{2r-1}, {\bold D}^{2r-1};\Q)^{\alt}\simeq H^{2r}(\square^{2r-1};\Q)^{\alt}.$$ Since 
$H^j(\square^n;\Q)^{\alt}=\{0\}$ if $n>0$, there is a chain
\newline $\ga_{2r-1}\in AC_{2r-1}(P^{2r-1}, {\bold D}^{2r-1};\Q)^{\alt}$
such that $\delta\ga_{2r-1}=Z.$
\,\,  We  have the equality 
\[\delta\part \ga_{2r-1}=\part\delta\ga_{2r-1}=\part Z=0\]
so the chain $\part \ga_{2r-1}$ defines a homology class in 
$H_{2r-3}(P^{2r-2}, {\bold D}^{2r-2};\Q)^{\alt}$ which is 0 if $2r-2>0$. Inductively one can
find chains $\ga_i\in AC_i(P^i, \bold D^i;\Q)$ for $ i=1,\cdots 2r-1$ such that
\[\part \ga_{i+1}=\delta \ga_i\quad (1\leq i\leq 2r-2)\quad \text{and}\quad
\part \ga_1=0.\]
There are elements $s(i)\in \ZZ/2\ZZ$ for $i=1,\cdots, 2r-1$ such that 
for the chains $c_i=(-1)^{s(i)}\ga_i$ one has
\[(-1)^{2r-1}\delta c_i=-Z,\quad \part c_{i+1}=-(-1)^i\delta c_i\quad (1\leq i\leq 2r-2)\quad \text{and}\quad
\part c_1=0.\]
Let $G=\sum_{i=1}^{2r-1} c_i.$  Then we have
\[\begin{aligned}
&dG\\
=&\sum_{i=1}^{2r-1} (\part+(-1)^i\delta)c_i\\
=&-Z.
\end{aligned}
\]
The elements $e_0+e_{-r}\otimes G$ and $e_{-r}$ are cocycles in $M_b$ and so
they form a basis of $H^0(M_b).$ Their images under the comparison map $c$
are
\[c(e_0+e_{-r}\otimes G)=e_0+e_{-r}(2\pi i)^r I(G)\quad\text{and} \quad
c(e_{-r})=e_{-r}(2\pi i)^r.\]
By \cite{Ki} Corollary 4.9 $I(G)$ is equal to the Abel-Jacobi image of $Z$. 
$I(G)$ is well defined modulo  integrals of the form 
$\omega_{2r-1}$ on cycles in 
$$H^{2r-1}(P^{2r-1}, \bold H^{2r-1};\Q)^{\alt}=H_{2r-1}(P^{2r-1}-\bold H^{2r-1};\Q)^{\alt}$$
where $ \bold H^{2r-1}$ is the union  of the faces of $P^{2r-1}.$
So it is well defined modulo $\QQ$. Hence $(2\pi i)^r I(G)$ defines an element
in $\C/(2\pi i)^r \Q$ which is the regulator of the element $z$.

\end{proof}

\section{The Hodge realization of Polylog motives}
\label{polylog}
We recall the definition of Polylog motives from \cite{BK}.
Let $a$ be an element of $F-\{0,1\}$ and $k$ be a positive integer.
In \cite{B2} Bloch defines an element 
$\rho_k(a)\in Z^k(\Spec F, 2k-1)^{\alt}=N^1_k$ described as follows.
It is $(-1)^{\frac{k(k-1)}2}$ times the locus in $\square^{2k-1}$ paremetrized in inhomogenious coordinates
by
\[(x_1,\cdots, x_{k-1}, 1-x_1, 1-\frac{x_2}{x_1},\cdots, 1-\frac{x_{k-1}}{x_{k-2}}, 1-\frac{a}{x_{k-1}})\]
where $x_1,\cdots x_{k-1}\in \square^1.$
In this section we omit the alternation symbol.   Then one has 
\[\part\rho_k(a)=-(a)\cdot \rho_{k-1}(a)\]
and 
\[[\rho_k(a)]+[(a)|\rho_{k-1}(a)]+[(a)|(a)|\rho_{k-2}(a)]+\cdots +[(a)|\cdots |(a)|\rho_1(a)]\]
is a cocycle in $B(N)$ of degree 0. So it defines an element $\bold{Li}_k(a)$
of $(\chi_N)_k$. We define the comodule $M_k(a)$ of $\chi_N$ as follows.
As an Adams graded $\Q$-vector space 
$$ M_k(a)=\bigoplus_{j=0}^k \Q e_{-j}$$ where $e_{-j}$
is a generator of the Adams degree $(-j)$-subspace of $M_k(a).$  The coaction $\Delta$ is defined by the equalities
\[\Delta e_0=e_0\otimes 1+e_{-1}\otimes \bold{Li_1}(a)+e_{-2}\otimes \bold{Li}_2(a)+
\cdots +e_{-k}\otimes \bold{Li}_k(a)\]
and 
\[\Delta e_{-j}=e_{-j}\otimes 1+e_{-j-1}\otimes[ (a)]+e_{-j-2}\otimes [(a)|(a)]
+\cdots +e_{-k}\otimes [(a)|\cdots |(a)]\]
for $1\leq j\leq k$. 
Up to shifting the Adams grading by $-k$, $M_k(a)$ is isomorphic to the sub comodule in $\chi_N$ generated by
\[e_0=\bold{Li}_k(a),\,\,\text{ and } e_{-j}=[(a)|\cdots |(a)]=(k-j)\text{-th power of }(a)
\text{ for } 1\leq j\leq k.\]
Let $[a]$ be the extension of $\Q$ by $\Q(1)$ defined by the element 
$$a\in CH^1(\Spec F, 1)=\text{Ext}^1_{\rm MT_{BK}}(\Q,\Q(1)).$$
Then $M_k(a)$ is an extension
of $\Q$ by $\text{Sym}^{k-1} ([a])(1)$, and $M_1(a)=[1-a].$ 
We first compute the image $\Psi(M_k(a))$ under the functor $\Psi$. For this it is easier
to assume that the elements $a$ and $1-a$ are linearly independent in
$H^1(N_1)=CH^1(\Spec F, 1)=F^\times\otimes \Q$. The reason is as follows.
If we look at the construction of a 1-minimal model  $N\{1\}$ of $N$
which is explained for example in the proof of \cite{KM} Part IV Theorem 3.7, 
we see under the above assumption that there exists a 1-minimal model 
$N\{1\}$ of $N$ whose image in $N$  contains $a$, $1-a$ and $\rho_k(a)$ for $k\geq 2$.
Otherwise we will compute 
$\Phi(M_k(a))$ later. As an object of $\text{Conn}^{0f}_{N\{1\}}$ the connection
$\Ga$ of $M_k(a)$ is defined as follows.
\[
\begin{aligned}
\Ga(e_0)=&e_{-1}\otimes \rho_1(a)+e_{-2}\otimes \rho_2(a)+\cdots +
e_{-k}\otimes \rho_k(a),\\
\Ga(e_{-j})=&e_{-j-1}\otimes (a)\,\,(1\leq j\leq k-1),\quad \Ga(e_{-k})=0.
\end{aligned}
\] 
We define the chains $\eta_k(i)\in AC_{k+i}(P^{k+i},\bold D^{k+i};\Q)^{\alt}$
for $0\leq i\leq k-1$. Originally $\eta_k(i)$ are defined  in \cite{BK}.
Consider the chain defined as  the locus of
\[
\begin{aligned}
&(x_1,\cdots, x_i,t_{i+1},\cdots, t_{k-1}, 1-x_1,\cdots, 1-\frac{x_i}{x_{i-1}},1-\frac{t_i}{x_i}),\\
&t_{k-1}\in [0,a], t_{k-2}\in [0,t_{k-1}],\cdots ,t_i\in [0, t_{i+1}],\\
&x_1,\cdots, x_i\in \PP^1(\C).
\end{aligned}
\]
Here more precisely one chooses a path $\ga$ from 0 to $a$ in $\PP^1$,
and $t_j$ should be denoted by $\ga(t_j)$. The orientation of this chain is defined so that $(t_i,t_{i+1},\cdots, t_{k-1})$ are positive  real coordinates. We define $\eta_k(i)$ 
 as $(-1)^{\frac{(k-1)(k-2)}2+i}$ times this chain.
We have the following identities.
\begin{equation}
\label{boundary}
\begin{aligned}
&(-1)^{2k-1}\delta\eta_k(k-1)=-\rho_k(a),\\
&(-1)^{k+i}\delta\eta_k(i)=-\part \eta_k(i+1)-(a)\cdot \eta_{k-1}(i)\quad (0\leq i\leq k-2 \text{ and } k\geq 2),\\
&\part \eta_k(0)=0.
\end{aligned}
\end{equation}
Note that all the boundary terms except $t_i=t_{i+1}$ and $t_{k-1}=a$ vanish
by alternation.
Let $\xi_k(a)\in AC^0$ be the sum $\dis \sum_{i=0}^{k-1}\eta_k(i).$
By (\ref{boundary}) we have
\[
\begin{aligned}
d\xi_k(a)=&-\rho_k(a)-(a)\cdot \xi_{k-1}(a)\,\, \text{ for }\,k\geq 2,\\
 d\xi_1(a)=&-\rho_1(a).
 \end{aligned}\]
 Let $Z_k(a)$ be the element 
 \[e_0\otimes 1+e_{-1}\otimes \xi_1(a)+e_{-2}\otimes \xi_2(a)+
 \cdots +e_{-k}\otimes \xi_k(a)\]
 of $(M_k(a))_b$ of cohomological degree 0. Then we have the equality
 \[\begin{aligned}
 &dZ_k(a)\\
 =&de_0\otimes 1+\sum_{j=1}^{k}d(e_{-j}\otimes \xi_j(a))\\
 =& \sum_{j=1}^ke_{-j}\otimes \rho_j(a)\\
 +&
 \sum_{j=1}^{k}e_{-j-1}\otimes (a)\cdot \xi_j(a)+e_{-j}\otimes
 (-\rho_j(a)-(a)\cdot \xi_{j-1}(a))\,\,(\text{ where } \,\,e_{-k-1}=0 \text{ and }\,\,\xi_0(a)=0)\\
 =&0.
 \end{aligned}
 \]
 So $Z_k(a)$ is an element of $H^0((M_k(a))_b).$  Let $p$ be a path on $\PP^1(\C)$
 from 1 to $a$ which avoids 0. $p$  defines an element of $AC^0.$
 Consider the chains $p^{\circ j}$ for $j\geq 0$.  $p^{\circ 0}=1$ by convention.
 By Proposition \ref{bullet} we see by induction on $j$ that
 \[d(p^{\circ j})=j dp\circ p^{\circ(j-1)}=-j (a)\cdot p^{\circ( j-1)}\,\,(\text{ where } \, p^{\circ( -1)}=0).\]
 For $j=1,\cdots, k$ let $L_j(a)$ be the element 
 $\dis \sum_{t=j}^k e_{-t}\otimes \frac1{(t-j)!}
 p^{\circ (t-j)}$
 of $(M_k(a))_b$ of cohomological
 degree 0.  We see that
 \[
 \begin{aligned}
 & dL_j(a)\\
 =&\sum_{t=j}^ke_{-t-1}\otimes \frac1{(t-j)!}(a)\cdot p^{\circ(t-j)}
 -e_{-t}\otimes \frac1{(t-j-1)!}(a)\cdot p^{\circ(t-j-1)}\\
& (\text{where } \,e_{-k-1}=0\,\,\text{ and }\,\,p^{\circ (-1)}=0)\\
=&0.
 \end{aligned}\]
 So $L_j(a)$ is an element of $H^0((M_k(a))_b)$.  We see $Z_k(a)$
 and $L_j(a)$ for $1\leq j\leq k$ is a basis of $H^0((M_k(a))_b)$.
 We compute $c(Z_k(a))$.  The integral 
 $\dis \int_{\eta_k(i)}\omega_{k+i}=0$ for $i>0$ for reasons of type,
 and
 \[
 \begin{aligned}
& I(\eta_k(0))\\
=&(-1)^{\frac{k(k-1)}2+\frac{(k-1)(k-2)}2}\frac1{(2\pi i)^k}
\int_{0\leq t_0\leq t_1\leq \cdots \leq t_{k-1}\leq a}
\frac{dz_1}{z_1}\wedge \cdots \wedge \frac{dz_k}{z_k}\\
=&(-1)^{\frac{k(k-1)}2+\frac{(k-1)(k-2)}2+k-1}\frac1{(2\pi i)^k}
\int_{0\leq t_0\leq t_1\leq \cdots \leq t_{k-1}\leq a}
\frac{d(1-t_0)}{1-t_0}\frac{dt_1}{t_1}\cdots  \frac{dt_{k-1}}{t_{k-1}}
\\
=&\frac{-1}{(2\pi i)^k}Li_k(a).
\end{aligned}
\]
It follows that
\[c(Z_k(a))=e_0-\sum_{j=1}^k e_{-j}Li_j(a).\]
Similarly  we see that
\[
\begin{aligned}
c(L_j(a))\\
=&\sum_{t=j}^k e_{-t} \frac{(2\pi i)^t}{(t-j)!}
 I(p^{\circ (t-j)})\\
 =& \sum_{t=j}^k e_{-t} \frac{(2\pi i)^t}{(t-j)!}
 \frac{(\log a)^{t-j}}{(2\pi i)^{t-j}}\\
=& \sum_{t=j}^k e_{-t} \frac{(2\pi i)^j}{(t-j)!}(\log a)^{t-j}.
\end{aligned}
\]  The period matrix of $\Psi(M_k(a))$  for the de Rham basis
$e_0,\cdots, e_{-k}$ is equal to 
\[
\left(
\begin{array} {ccccc}
1&&&&\\
-Li_1(a)& 2\pi i &&&\\
-Li_2(a)& 2\pi i\log a& (2\pi i)^2&&\\
\vdots& 2\pi i \frac{(\log a)^2} {2!}& (2\pi i)^2\log a&\\
\vdots& \vdots &&\ddots&\\
-Li_k(a)&\cdots &\cdots && (2\pi i)^k
\end{array}
\right).
\]
As for the image $\Phi(M_k(a))$ of the functor $\Phi$, we can take
\[
\begin{aligned}
&\Delta Z_k(a)\\
=&\Delta(e_0)\otimes 1+\Delta(e_{-1})\otimes \xi_1(a)+\Delta(e_{-2})\otimes \xi_2(a)+
 \cdots +\Delta(e_{-k})\otimes \xi_k(a)
 \end{aligned}
 \]
 and
 $$\Delta(L_j(a))=\sum_{t=j}^k \Delta(e_{-t})\otimes \frac1{(t-j)!}
 p^{\circ (t-j)}$$ for $j=1,\cdots, k$ as the basis of the betti part.
 The basis of the de Rham part is $\Delta(e_{-j})$  for $j=0,\cdots, k$.
 The period matrix is the same as that of $\Psi(M_k(a)).$

\vskip 0.5cm

\noindent Kenichiro Kimura

\vskip 0.1cm 

\noindent Department of Mathematics,

\noindent University of Tsukuba,

\noindent 1-1-1 Tennodai,

\noindent Tsukuba, Ibaraki,

\noindent 305-8571

\noindent Japan.

\printindex

\end{document}